\documentclass{amsart}
\usepackage[utf8]{inputenc}
\usepackage{verbatim}

\usepackage{amssymb}
\usepackage{soul,xcolor}
\setstcolor{red}

\usepackage{dynkin-diagrams}
\usepackage{latexsym,amssymb,amsmath}
\usepackage{epsfig}
\input xy
\xyoption{all}
\usepackage[english]{babel}

\usepackage{amsmath}
\newtheorem{theorem}{Theorem}[section]
\newtheorem{lemma}[theorem]{Lemma}
\newtheorem{prop}[theorem]{Proposition}
\newtheorem{coro}[theorem]{Corollary}

\newtheorem{remark}[theorem]{Remark}

\def\PP{\mathbb{P}}\def\AA{\mathbb{A}}\def\RR{\mathbb{R}}
\def\CC{\mathbb{C}}\def\HH{\mathbb{H}}\def\OO{\mathbb{O}}
\def\ZZ{\mathbb{Z}}\def\QQ{\mathbb{Q}}\def\LL{\mathbb{L}}
\def\JJ{\mathbb{J}} \def\XX{\mathbb{X}}\def\SS{\mathbb{S}}
\def\cA{{\mathcal A}}\def\cY{{\mathcal Y}}\def\cZ{{\mathcal Z}}
\def\cB{{\mathcal B}}
\def\cC{{\mathcal C}}
\def\cD{{\mathcal D}}
\def\cE{{\mathcal E}}
\def\cF{{\mathcal F}}
\def\cG{{\mathcal G}}
\def\cI{{\mathcal I}}
\def\cJ{{\mathcal J}}
\def\cL{{\mathcal L}}
\def\cQ{{\mathcal Q}}
\def\cM{{\mathcal M}}
\def\cO{{\mathcal O}}
\def\cS{{\mathcal S}}
\def\cU{{\mathcal U}}
\def\cX{{\mathcal X}}
\def\ra{\rightarrow}\def\lra{\longrightarrow}
\def\fg{{\mathfrak g}}\def\fh{{\mathfrak h}}
\def\fso{{\mathfrak{so}}}\def\ft{{\mathfrak{t}}}\def\fsl{{\mathfrak{sl}}}\def\fsp{{\mathfrak sp}}
\def\s{{\sigma}}\def\t{{\tau}}\def\iot{{\iota}}
\def\fn{{\mathfrak n}}\def\fe{{\mathfrak e}}\def\ff{{\mathfrak f}}
\def\fk{{\mathfrak k}}\def\fg{{\mathfrak g}}
\def\fj{{\mathfrak j}}\def\fp{{\mathfrak p}}\def\fc{{\mathfrak c}} 
\def\fz{{\mathfrak z}} 
\def\fs{{\mathfrak s}}
\def\fS{{\mathfrak S}}
\def\fD{{\mathfrak D}}
\DeclareMathOperator{\Ho}{H}
\DeclareMathOperator{\Aut}{Aut}
\DeclareMathOperator{\Hom}{Hom}
\DeclareMathOperator{\SL}{SL}
\DeclareMathOperator{\SU}{SU}
\DeclareMathOperator{\GL}{GL}
\DeclareMathOperator{\PGL}{PGL}
\DeclareMathOperator{\PSL}{PSL}
\DeclareMathOperator{\SO}{SO}
\DeclareMathOperator{\Spin}{Spin}
\DeclareMathOperator{\PSO}{PSO}
\DeclareMathOperator{\Sp}{Sp}
\DeclareMathOperator{\Stab}{Stab}
\DeclareMathOperator{\id}{id}
\DeclareMathOperator{\ad}{ad}
\DeclareMathOperator{\Ad}{Ad}
\DeclareMathOperator{\Fl}{Fl}
\DeclareMathOperator{\OGr}{OGr}
\DeclareMathOperator{\IG}{IG}
\DeclareMathOperator{\IFl}{IFl}
\DeclareMathOperator{\OG}{OG}
\DeclareMathOperator{\OFl}{OFl}
\DeclareMathOperator{\rank}{rank}
\DeclareMathOperator{\diag}{diag}
\DeclareMathOperator{\Hilb}{Hilb}
\DeclareMathOperator{\Lie}{Lie}
\DeclareMathOperator{\Out}{Out}
\DeclareMathOperator{\Sym}{Sym}
\DeclareMathOperator{\End}{End}
\DeclareMathOperator{\pt}{pt} 
\DeclareMathOperator{\corank}{corank}
\DeclareMathOperator{\Ker}{Ker}
\DeclareMathOperator{\codim}{codim}
\DeclareMathOperator{\codegree}{codegree}
\DeclareMathOperator{\tr}{tr}
\def\GG{\mathbb{G}}

\usepackage{xcolor}
\def\vlad#1{\textcolor{blue}{{\bf Vlad:} #1 {\bf }}}
\def\laurent#1{\textcolor{red}{{\bf *** Laurent:} #1 {\bf ***}}}

\begin{document}

\author{Vladimiro Benedetti}

\author{Laurent Manivel}
%\affil{D\'epartement de math\'ematiques et applications, ENS, CNRS, PSL University, 75005 Paris, France}
%\affil{Institut de Math\'ematiques de Toulouse, UMR 5219, Universit\'e de Toulouse, CNRS, UPS IMT F-31062 Toulouse Cedex 9, France.}

%\title{A projective duality formula \linebreak for theta representations}
\title{Discriminants of  theta-representations}
%\date{}

\maketitle 

\begin{abstract}
Tevelev has given a remarkable  explicit formula for the discriminant of a complex simple Lie algebra, which can be 
defined as the equation of the dual hypersurface  of the minimal nilpotent orbit, or of the so-called adjoint variety.  
In this paper we extend this  
formula to the setting of graded Lie algebras, and express the equation of the corresponding dual hypersurfaces in terms 
of the reflections in the little Weyl groups, the associated complex reflection groups. This explains for example why the codegree of the 
Grassmannian $G(4,8)$ is equal to the number of roots of $\fe_7$. 
\end{abstract}

\section{Introduction}

Projective duality is a very classical construction in projective geometry that has been studied extensively in various contexts (see 
\cite{GKZ} and references therein), 
from the 19th century Pl\"ucker formulas for plane curves, to the exciting modern developments of homological projective duality
\cite{kuz}. 
Recall the classical definition: for  $X\subset \PP^r$ an embedded projective variety, the dual variety $X^\vee\subset \check{\PP}^r$ is 
the closure of the set of hyperplanes that are tangent to $X$ at some smooth point. This is a true duality in the sense that 
$(X^\vee)^\vee=X$, although the dual variety is in general a hypersurface. It is not always easy to decide when this is not the case, 
i.e. when the dual variety is degenerate. When the dual variety is a hypersurface, it is notoriously difficult to find an equation, 
or even to compute explicitly the degree of such an equation, which we call the \emph{codegree} of $X$ (another classical terminology 
for this number is the \emph{class} of $X$). 

The Katz-Kleiman formula \cite[Theorem 3.4 and Formula 3.19]{GKZ} provides an expression for the codegree in terms of the Chern classes of $X$ and 
the hyperplane class $\lambda$, namely 
$$\deg(X^\vee)=\int_X\frac{c(\Omega_X)}{(1-\lambda)^2},$$
with the convention that $\deg(X^\vee)=0$ exactly when the dual variety $X^\vee$ is not a hypersurface. (When the dual 
variety is degenerate, a variant of this formula allows to compute its dimension, and its genuine degree \cite{GK}). 
Although simple and beautiful, this formula is difficult to put in practice in general; even worse, it often expresses the 
codegree as an alternate sum of big integers, and it is not always easy even to decide whether the result is 
positive, or zero. 

The Katz-Kleiman formula was nevertheless applied to Grassmannians in  \cite{lascoux}, yielding an "explicit" formula given 
as a big sum of Vandermonde determinants with alternating signs.
Using this formula, Lascoux computed that 
$$\deg(G(3,9)^\vee)=120, \qquad \deg(G(4,8)^\vee)=126.$$
Of course, this leaves completely open the question of finding explicit equations for the dual varieties 
of Grassmannians (for $G(3,9)$ and $G(4,8)$, this is discussed in \cite{HO}). 

A class of varieties whose projective duals admit a remarkably explicit description is that of  adjoint varieties of simple 
complex Lie groups. For each such group $G$, with Lie algebra $\fg$, the adjoint variety is defined as the unique closed $G$-orbit
$X_{ad}(\fg)$ inside $\PP(\fg)$. Its dual variety is a hypersurface whose equation is often called the discriminant of $\fg$. 
An elegant explicit formula for this discriminant was given by  Tevelev \cite{tevelev}, which through the Chevalley restriction theorem
reduces to a simple combinatorial formula in terms of the root system. In particular, the degree of the discriminant is simply
given by the number of long roots. 

For the other homogeneous varieties, the situation is much less clear. 
The question of deciding which ones have degenerate duals was discussed
in \cite{KM}. A positive formula for the codegree was given in \cite{cw},
but it seems difficult to make it explicit. More recently, localization techniques 
in equivariant cohomology  were used in \cite{fnr} to obtain the degrees
of the discriminants as explicit but complicated sums of rational numbers. 
Typically, formula (5.1) in loc. cit. is 
$$\deg(G(k,n)^\vee)
=\frac{2k}{n+1}\sum_{S}\prod_{i\in S,j\notin S}\frac{\ell(S)+j-i}{i-j},$$
where the sum is over the sets $S$ of $k$ integers between $1$ and $n$, 
and $\ell(S)=\sum_{s\in S}s$. These techniques provide nevertheless no insight on the discriminants themselves. 

The reader may have noticed that the codegrees $120$ of $G(3,9)\subset\PP(\wedge^3\CC^9)$ and $126$ of  $G(4,8)\subset\PP(\wedge^4\CC^8)$ are familiar numbers in Lie theory: the latter is the number of 
roots of $E_7$, the former is the number of positive roots in $E_8$. It is a remarkable fact that these two exceptional Lie algebras admit the simple models 
$$\fe_7=\fsl_8\oplus \wedge^4\CC^8,
\qquad \fe_8=\fsl_9\oplus \wedge^3\CC^9\oplus \wedge^6\CC^9.$$
This suggests that Tevelev's formula for the discriminants could admit an
extension to the setting of graded Lie algebras, a topic which was developped with great success, in particular by the Russian school, 
starting from the late sixties, in parallel with the theory of Kac-Moody algebras \cite{kac, vinberg}. For $G(3,9)$ and $G(4,8)$ this connection was
already observed in \cite{HO}. Here we study the general situation and 
show how to extend Tevelev's formula in the graded setting, for $\fg=\bigoplus_{i\in\ZZ_m}\fg_i$ a cyclically graded simple Lie algebra. 
Equivalently, $\fg$ admits a Lie algebra automorphism $\theta$, acting on $\fg_i$ by multiplication by $\xi^i$ for some 
$m$-th root of unity $\xi$. Our main result is the following. 

\begin{theorem}
Suppose that $m=2,3$ or $5$, that  the $\theta$-corank is zero, and that $\fg_0$ is semisimple and acts irreducibly on $\fg_1$. 
Then the dual variety of the unique closed orbit in $\PP (\fg_1)$ has an explicit equation, given on a 
Cartan  subspace by a product of roots. 
\end{theorem}

We refer to Section \ref{sec_main_theorem} for a more precise statement. The product of roots to be considered depends on the 
so-called little Weyl group of the graded 
Lie algebra, which is a complex reflection group; the roots in question define the reflection hyperplanes of the complex 
reflections in the Weyl group, at least on a Cartan subspace; the equation of the dual variety is then completely 
determined because of Vinberg's graded version of the Chevalley restriction theorem. The notion of $\theta$-corank is explained 
in Section \ref{sec_invariants} The restrictive hypothesis of the 
Theorem are justified in Section \ref{sec_counterexamples} For example, there are some natural gradings to consider for which $m=4$ or $m=6$, 
but for which the dual variety of the closed orbit in $\PP (\fg_1)$ is not a hypersurface.

\medskip

\paragraph{\bf{Outline of the paper.}} 
We start in Section \ref{sec_2} with some very classical preliminaries about complex simple Lie algebras, and we explain in 
Section \ref{sec_3} how they extend to the graded setting. Section \ref{sec_4} focuses on the associated Weyl groups, which in the 
graded setting are only complex reflection groups in general; we explain how to produce complex reflections 
explicitly. This allows to write down an equation of the dual variety in Section \ref{sec_5}, and prove the main Theorem. 
In Section \ref{sec_6} we discuss the relevant examples, either exceptional or classical. Finally, in the independent Section \ref{sec_spin_lagr_varieties} 
we briefly explain how to extend Lascoux's formula to Lagrangian Grassmannians and spinor varieties.

\medskip\noindent {\it Acknowledgements.} 
We acknowledge support from the ANR project FanoHK, grant
ANR-20-CE40-0023. The first author is partially supported by the EIPHI Graduate School (contract ANR-17-EURE-0002)

\section{Preliminaries on Lie algebras}
\label{sec_2}

\subsection{Simple Lie algebras}
Graded Lie algebras are generalizations of Lie algebras, so let us start with the latter and recall a few basic facts. 
If $\fg$ is a simple complex Lie algebra, a \emph{Cartan subalgebra} $\fh\subset \fg$  is a maximal abelian subspace  
made of semisimple elements. Since the elements of $\fh$ commute, they can be diagonalized simultaneously, and one gets a 
decomposition
\[
\fg=\fh \oplus \bigoplus_{\alpha\in R(\fg)}\fg_{\alpha},
\]
where $R(\fg)\subset \fh^\vee$ is the set of \emph{roots}, such that $[h,x]=\alpha(h)x$ for any $x\in \fg_\alpha$ and any $h\in\fh$. 
Each  \emph{root space} $\fg_\alpha$  is one-dimensional. 
\medskip

The linear span of $R(\fg)$ is $\fh^\vee$ and one can extract basis of $\fh^\vee$ from $R(\fg)$, made of \emph{simple roots}. 
Among other properties, if $\alpha_1,\ldots,\alpha_n$ is such a basis of simple roots, any root $\alpha\in R(\fg)$ can be written 
as $\alpha=\sum_i n^\alpha_i \alpha_i$, where the $n^\alpha_i$'s are integers, either all non-negative or  all non-positive. 
For any root $\alpha$, $-\alpha$ is also a root  and there exists $h_\alpha\in \fh$ such that $\fg_{-\alpha}\oplus \CC h_\alpha \oplus \fg_\alpha\cong \fsl_2$.

\subsection{Invariants} Given $\fh\subset \fg$ one defines the \emph{Weyl group} $W=W(R(\fg))$ as $N_G(\fh)/Z_G(\fh)$, where $N_G(\fh)$ is the normalizer of $\fh$ in $G$, and $Z_G(\fh)$ its  centralizer. This is a finite group $W$ generated  by reflections $s_\alpha$, $\alpha\in R(\fg)$, which fix the hyperplane orthogonal to $\alpha$ and send $\alpha$ to $-\alpha$ ($\fg$ and $\fg^\vee$ being identified via the Killing form). The Weyl group $W$ permutes the roots in $R(\fg)$, and acts transitively on the set of roots of a given length (in the simply laced case all roots have the same length, and by convention they will be considered as long;
otherwise there are two classes, long and short roots).  A celebrated result, usually referred to as  Chevalley's restriction theorem, is the following:
\begin{theorem}[Chevalley]
\label{thm_chevalley}
The natural restriction morphism gives an isomorphism of invariant rings
\[
\CC[\fg]^G=\CC[\fh]^W.
\]
Moreover $\CC[\fh]^W$ is a polynomial algebra.
\end{theorem}

The second statement in the theorem holds, more  generally, exactly  for those finite subgroups of $\GL_n$ which are \emph{complex reflection groups}, i.e. the finite groups generated by complex reflections.

\subsection{Tevelev's duality formula} If $G$ denotes the adjoint group of the simple Lie algebra $\fg$, recall that 
the \emph{adjoint variety} of $\fg$ is the (unique) closed $G$-orbit inside $\PP(\fg)$. Tevelev \cite{tevelev} shows that the projective 
dual variety (inside $\PP(\fg^\vee)\cong \PP(\fg)$, which are isomorphic by the Killing form) of the adjoint variety is an irreducible hypersurface  $\cD_\fg$, usually called the {\emph discriminant} of $\fg$, and given by some $G$-invariant  polynomial $D_\fg\in \CC[\fg]^G$ (defined up to constant). By Chevalley's restriction theorem, $D_\fg$ is completely determined by its restriction 
to $\fh$, for which Tevelev provides a simple expression. 

\begin{theorem}[Tevelev]\label{tevelev}
The dual variety of the adjoint variety $X\subset \PP(\fg)$  is the hypersurface $\cD_\fg$ defined by the $G$-invariant 
polynomial $D_\fg$ such that 
\[
D_\fg|_{\fh}:=\prod_{\alpha\in R(\fg)_l}\alpha,
\]
where $R(\fg)_l$ is the set of long roots.% (we assume that all roots are long in the simply laced case).
\end{theorem}

\begin{remark}
Tevelev's result is actually slightly more general. Indeed it shows that the dual variety of the orbit closure of the short root spaces inside $\PP(\fg)$ is also a hypersurface defined by a $G$-invariant polynomial $D_{\fg,s}$ whose restriction to $\fh$ is given by
\[
D_{\fg,s}|_{\fh}:=\prod_{\alpha\in R(\fg)_s}\alpha,
\]
where $R(\fg)_s$ is the set of short roots. Notice that in this case the orbit in question is not closed.
\end{remark}

The main goal of this paper is to extend  this statement  to the graded setting.

\section{Cyclically graded Lie algebras}
\label{sec_3}

For this section we follow essentially \cite{vinberg}. Let $m$ be a positive integer and $\ZZ_m$ denote the cyclic group of order $m$. 
A $\ZZ_m$-graded Lie algebra is a Lie algebra $\fg$ with a grading $\fg=\bigoplus_{i\in \ZZ_m}\fg_i$ such that $[\fg_i,\fg_j]\subset \fg_{i+j}$. To any such grading, one can associate an automorphism $\theta$ of order $m$ of $\fg$, acting on $\fg_i$  by multiplication by $\xi^i$, where 
$\xi$ is a primitive $m$-th root of unity. Conversely, the grading can of course be recovered from the automorphism.

The connected component $G_0$ of the $\theta$-invariant subgroup $G^\theta$ has Lie algebra $\fg_0$.  It acts on any $\fg_i$ since $[\fg_0,\fg_i]\subset \fg_i$. We will be mostly interested in the $G_0$-representation $\fg_1$, which is called a $\theta$-representation. 

\subsection{Inner gradings}
A first class of examples correspond to the case where $\theta$ is an inner automorphism of a simple Lie algebra $\fg$. Choose a root 
$\alpha_i$ in a basis $\alpha_1,\ldots,\alpha_n$ of simple roots in $R(\fg)$. Recall that  any root $\alpha\in R(\fg)$ can be decomposed as $\alpha=\sum_j n^\alpha_j\alpha_j$. If $m$ is the maximal possible value of $n^\alpha_i$, one  defines a $\ZZ_m$-grading on $\fg$ by letting 
\[
\fg_j:= \delta_{j,0}\fh \oplus \bigoplus_{\alpha\mid n^\alpha_i=j\; mod \;m}\fg_\alpha.
\]
In this case $\fg_0$ is semisimple,  thus uniquely defined by a Dynkin diagram which can be described as follows. Let $\Delta$ be the 
Dynkin diagram of $\fg$ (each node corresponding to a simple root), and $\hat{\Delta}$ the corresponding affine diagram. The Dynkin diagram of $\fg_0$ is then $\hat{\Delta}\setminus \{ \alpha_i \}$. Moreover (note that $\fh$ is a Cartan subalgebra of  $\fg_0$), $\alpha_i$ 
is a highest weight of the irreducible $G_0$-representation $\fg_1$.

\begin{remark}
This construction depends on the choice of a single node (or root) $\alpha_i$. Notice that a grading could also be defined by the choice of more than one node (with some non-negative multiplicities); however, in the following, we will be mainly interested in graded Lie algebras such that $\fg_0$ is semisimple, and those arise only when one chooses one single node (see \cite{vinberg}).
\end{remark}

\subsection{Kac's classification} Even when $\theta$ comes from an outer automorphism of $\fg$,  Kac \cite{kac} was able to show that there exists an extended Dynkin diagram (playing the role played by $\hat{\Delta}$ for inner automorphisms) from which one can recover $\fg_0$ and $\fg_1$ exactly as before. These extended Dynkin diagrams are given in \cite{vinberg}. They exist only for Lie algebras  admitting outer automorphisms and are of the following types: $A_n^{(2)}$, $D_n^{(2)}$, $E_6^{(2)}$, $D_4^{(3)}$, where the superscript in parenthesis denotes the order of $\theta$ modulo inner automorphisms. We include these diagrams just below, as well as the affine Dynkin diagrams that correspond
to inner automorphisms. In analogy to the inner automorphism case, each node plays the role of the root $\alpha_i$ above, and the index near the node plays the role of the corresponding integer $m$. %The superscript in parenthesis indicates the order of the outer automorphism whose class $\theta$ belongs to.
\medskip

\begin{center}
$
A^{(1)}_{n-1}=\dynkin[extended, labels={1,1,1,1,1}]A{}, \
B^{(1)}_{n-1}=\dynkin[labels={1,1,2,,2,2,2}]B[1]{}, \
C^{(1)}_{n-1}=\dynkin[extended, labels={1,2,2,2,2,1}]C[1]{}, 
$
$
D^{(1)}_{n-1}=\dynkin[extended, labels={1,1,2,,,2,1,1}]D[1]{}, \qquad \
F^{(1)}_{4}=\dynkin[extended, labels={1,2,3,4,2}]F4, \qquad \
G^{(1)}_{2}=\dynkin[extended, labels={1,2,3}]G2, 
$
$
E^{(1)}_{6}=\dynkin[extended, labels={1,1,2,2,3,2,1}]E6, \
E^{(1)}_{7}=\dynkin[extended, labels={1,2,2,3,4,3,2,1}]E7, \
E^{(1)}_{8}=\dynkin[extended, labels={1,2,3,4,6,5,4,3,2}]E8, 
$
$
A^{(2)}_{2}=\dynkin[labels={4,2}]A[2]2, \
A^{(2)}_{2n-2}=\dynkin[labels={4,4,,4,4,4,2}]A[2]{even}, \
A^{(2)}_{2n-3}=\dynkin[labels={2,2,4,4,4,4,4,2}]A[2]{odd}, 
$
$
D^{(2)}_{n}=\dynkin[labels={2,2,2,2,2,2}]D[2]{}, \qquad \
E^{(2)}_{6}=\dynkin[labels={2,4,6,4,2}]E[2]{6}, \qquad  \
D^{(3)}_{4}=\dynkin[labels={3,6,3}]D[3]{4}.
$
\end{center}

\subsection{Invariants}
\label{sec_invariants}
Remarkably, Chevalley's restriction theorem can be extended to the graded setting. 

\medskip

\paragraph{\bf{Definition (Cartan subspace, little Weyl group).}} A maximal subspace $\fc\subset\fg_1$ containing only semisimple and commuting elements
of $\fg$ is called a \emph{Cartan subspace}. All Cartan subspaces are conjugate under $G_0$, hence of the same  dimension, called the 
$\theta$-$\rank $ of $\fg$ \cite{vinberg}. Letting $N_0(\fc):=\{g\in G_0 \mid g(\fc)=\fc\}$ and  $Z_0(\fc):=\{g\in G_0 \mid g|_{\fc}=\id_{\fc}\}$,
the Weyl group of the graded Lie algebra is defined as $W_\theta:=N_0(\fc)/Z_0(\fc)$. It is usually referred to as the \emph{little Weyl group} (see \cite[Section 7]{GRLY}) to distinguish it from the Weyl group of $\fg$.

\begin{theorem}[Vinberg]
\label{thm_chevalley_vinberg}
The natural restriction morphism gives an isomorphism of invariant rings
\[
\CC[\fg_1]^{G_0}=\CC[\fc]^{W_\theta}.
\]
Moreover $\CC[\fc]^{W_\theta}$ is a polynomial algebra.
\end{theorem}

\paragraph{\bf{Definition ($\theta$-torus, $\theta$-$\corank$).}} Let $\fg$ be a simple $\ZZ_m$-graded Lie algebra of rank $r$ and $\theta$-$\rank=r_1$.
Following \cite{vinberg}, the smallest algebraic subalgebra of $\fg$ containing a given Cartan subspace $\fc\subset \fg_1$ is a Lie subalgebra $\ft$ which is the Lie algebra of a torus $T\subset G$ of dimension $\dim(T)=r_1 \varphi(m)$, where $\varphi$ is Euler's function.  More precisely, one shows that $\ft=\bigoplus_{(i,m)=1}\ft_i$, where $\ft_i=\ft\cap \fg_i$, and in this decomposition each term has dimension $r_1$. The torus $T$ is the \emph{$\theta$-torus} associated to $\fc$, and can be recovered as the closure of $\exp(\fc)$, where $\exp:\fg\to G$ is the exponential map. In 
particular $r_1\varphi(m) \leq r$. We  define the  $\theta$-$\corank$ of $\fg$ as the integer $r-r_1\varphi(m)$.

\medskip

\section{Simple Lie algebras with prime grading and maximal $\theta$-rank}
\label{sec_4}

From now on we  assume that $m$ is prime and that the $\theta$-$\corank$ of $\fg$ vanishes; we say in that case that $\fg$ has maximal $\theta$-rank. 
Under these conditions, the rank of $\fg$ is equal to $(m-1)r_1$, where $r_1$ is the $\theta$-rank. 
\medskip

\paragraph{\bf{Graded simple Lie algebras of maximal $\theta$-rank.}} The action of $G_0$ on $\fg_1$ is given by $\rho_1:G_0\to \GL(\fg_1)$. Graded Lie algebras such that $\rho_1(G_0)$ is semisimple were classified in \cite{vinberg}; these are the graded algebras constructed by choosing one single root from one of the extended diagrams. In Table \ref{tab_exc_graded_vinberg} we reported all such graded algebras with prime grading such that $\theta$-$\corank=0$. For the exceptional cases (included $D_4^{(3)}$) the result is obtained just by using Vinberg's results in \cite{vinberg}, so let us explain the few classical ones. Notice that all the nodes in the classical cases have index either equal to $1$, $2$ or $4$; thus the only nodes we will take into account are those with index equal to $2$, and they will correspond to some well-known symmetric spaces ($m=2$).
\begin{itemize}
\item[$B_{n-1}^{(1)}$:] The choice of one of the roots with index $2$ gives a grading of $\fso_{2n-1}$ such that $(\fso_{2n-1})_0=\fso_{2p}\times \fso_{2q+1}$ and $(\fso_{2n-1})_1=\CC^{2p}\otimes \CC^{2q+1}$, with $q=n-1-p$ and $2\leq p \leq n-1$. This corresponds to the symmetric space denoted by $\SO(2p+2q+1)/\SO(2p)\times \SO(2q+1)$ in \cite{helgason}, and its $\theta$-rank is equal to $\min\{2p,2q+1\}$. Since the rank in this case is $n-1$, we have $\theta$-$\corank=0$ if and only if either $2p=n-1=2q$ or $2p=n=2q+2$; 

\item[$C_{n-1}^{(1)}$:] The choice of one of the roots with index $2$ gives a grading of $\fsp_{2n-2}$ such that $(\fsp_{2n-2})_0=\fsp_{2p}\times \fsp_{2q}$ and $(\fsp_{2n-2})_1=\CC^{2p}\otimes \CC^{2q}$, with $q=n-1-p$ and $1\leq p \leq n-2$. This corresponds to the symmetric space denoted by $\Sp(p+q)/\Sp(p)\times \Sp(q)$ in \cite{helgason}, and its $\theta$-rank is equal to $\min\{p,q\}$. Since the rank in this case is $n-1$, $\theta$-$\corank=0$ is never satisfied; 

\item[$D_{n-1}^{(1)}$:] The choice of one of the roots with index $2$ gives a grading of $\fso_{2n-2}$ such that $(\fso_{2n-2})_0=\fso_{2p}\times \fso_{2q}$ and $(\fso_{2n-2})_1=\CC^{2p}\otimes \CC^{2q}$, with $q=n-1-p$ and $2\leq p \leq n-3$. This corresponds to the symmetric space denoted by $\SO(2p+2q)/\SO(2p)\times \SO(2q)$ in \cite{helgason}, and its $\theta$-rank is equal to $\min\{2p,2q\}$. Since the rank in this case is $n-1$, we have $\theta$-$\corank=0$ if and only if $2p=2q=n-1$; 

%\item[$A_{2}^{(2)}$:] this case is actually the same as the following one;

\item[$A_{2n-2}^{(2)}$:] The choice of the only root with index $2$ gives a grading of $\fsl_{2n-1}$ such that $(\fsl_{2n-1})_0=\fso_{2n-1}$ and $(\fsl_{2n-1})_1=S^{\langle 2 \rangle}\CC^{2n-1}$. This corresponds to the symmetric space denoted by $\SU(2n-1)/\SO(2n-1)$ in \cite{helgason}, and its $\theta$-rank is equal to $2n-2$. Since the rank in this case is $2n-2$, we have $\theta$-$\corank=0$; 

\item[$A_{2n-3}^{(2)}$:] The choice of the last root with index $2$ gives a grading of $\fsl_{2n-2}$ such that $(\fsl_{2n-2})_0=\fso_{2n-2}$ and $(\fsl_{2n-2})_1=S^{\langle 2 \rangle}\CC^{2n-2}$. This corresponds to the symmetric space denoted by $\SU(2n-2)/\SO(2n-2)$ in \cite{helgason}, and its $\theta$-rank is equal to $2n-3$. Since the rank in this case is also  $2n-3$, we have $\theta$-$\corank=0$. 

The choice of one of the first two roots with index $2$ gives a grading of $\fsl_{2n-2}$ such that $(\fsl_{2n-2})_0=\fsp_{2n-2}$ and $(\fsl_{2n-2})_1=\wedge^{\langle 2 \rangle}\CC^{2n-2}$. This corresponds to the symmetric space denoted by $\SU(2n-2)/\Sp(n-1)$ in \cite{helgason}, and its $\theta$-rank is equal to $n-2$. So $\theta$-$\corank=0$ is never satisfied;

\item[$D_{n}^{(2)}$:] The choice of one of the roots gives a grading of $\fso_{2n}$ such that $(\fso_{2n})_0=\fso_{2p+1}\times \fso_{2q+1}$ and $(\fso_{2n})_1=\CC^{2p+1}\otimes \CC^{2q+1}$, with $q=n-1-p$ and $0\leq p\leq n-1$. This corresponds to the symmetric space denoted by $\SO(2p+2q+2)/\SO(2p+1)\times \SO(2q+1)$ in \cite{helgason}, and its $\theta$-rank is equal to $\min\{ 2p+1,2q+1 \}$. Since the rank in this case is $n$, $\theta$-$\corank=0$ if and only if $2p+1=2q+1=n$.
\end{itemize}

\subsection{Homogeneous decomposition}
\label{sec_homogeneous_decomposition}

Let $\fc\subset \fg_1$ be a Cartan subspace, with associated $\theta$-torus $T$ and $\Lie(T)=\ft=\bigoplus_{i=1}^{m-1}\ft_i$. Since the $\theta$-$\corank$ vanishes, 
$T$ is a maximal torus in $G$. The corresponding decomposition  $\fg=\ft\oplus \bigoplus_{\alpha\in R(\fg)}\fg_\alpha$  is not $\theta$-stable, 
though. Our aim in this section is to construct a $\theta$-homogeneous decomposition. 
\medskip

\paragraph{\bf{The decomposition.}} For $\alpha\in R(\fg)$, let us decompose $\alpha=\sum_{i=1}^{m-1} \alpha_i$, where $\alpha_i=\alpha|_{\ft_i}$.

\begin{lemma}
$\alpha_i\neq 0$ for any $1\leq i \leq m-1$.
\end{lemma}
\begin{proof}
Suppose for instance that $\alpha_1=0$. Then $\fc=\ft_1\subset \alpha^\perp \subset \ft$. But $\alpha^\perp$ is an algebraic subalgebra of $\fg$ (meaning that it is the Lie algebra of a subtorus of $T$). Since $\ft$ is the smallest algebraic subalgebra containing $\fc$, we get a contradiction.
\end{proof}

Consider then $x\in\fg_\alpha$. By decomposing $x=\sum_i x_i$, one can check that for any $t\in \ft$, $[t,\theta(x)]=(\theta(\alpha)(t))x$, where $\theta(\alpha)=\sum_j \xi^j\alpha_j$. Thus $\theta$ acts on the roots of $\fg$, each orbit having cardinality $m$. Note also that all the 
roots in a given $\theta$-orbit have the same length. Indeed, since $\fg_i$ is orthogonal to $\fg_j$ with respect to the Killing form 
when $i+j\ne 0$, one has  $(\alpha,\alpha)=\sum_{i\in \ZZ_m} (\alpha_i,\alpha_{-i})=(\theta(\alpha),\theta(\alpha))$. 

 \smallskip
Let $[R(\fg)]$ denote  the set of $\theta$-orbits in $R(\fg)$. Then
\[
\fg=\ft \oplus \bigoplus_{[\alpha]\in [R(\fg)]}\fg_{[\alpha]}
\]
is a homogeneous decomposition of $\fg$, where  $\fg_{[\alpha]}:=\bigoplus_{i}\fg_{\theta^i(\alpha)}$. 

\begin{lemma} \label{lem_homogeneous_component}
$\fg_{[\alpha]}=\bigoplus_{0\leq i\leq m-1}\fg_{[\alpha],i}$, where $\fg_{\fg_{[\alpha],i}}:=\fg_{[\alpha]}\cap \fg_i$ is one-dimensional. 
\end{lemma}
\begin{proof}
Consider $x=\sum_j x_j\in \fg_\alpha$, with $x_j\in \fg_j$. Since $\alpha_1\neq 0$ and $[\fg_1,\fg_j]\subset [\fg_{j+1}]$, for $t\in\fc=\ft_1$ one obtains that $[t,x]=\sum_j \alpha_1(t) x_{j+1}$, and thus $x_j\neq 0$ for $0\leq j\leq m-1$. Then, by computing $\theta^i(x)$ explicitly, one verifies that $\fg_{[\alpha],j}=\CC x_j$.
\end{proof}

Similarly, one can prove the following:

\begin{lemma}
Let $t_\alpha=\sum_{i=1}^{m-1} t_{\alpha_i}$ be the coroot in $\ft$ of $\alpha=\sum_{i=1}^{m-1} \alpha_i$. Then $$\Big(\bigoplus_l t_{\theta^l(\alpha)}\Big)\cap \fg_i=\CC t_{\alpha_i}.$$
\end{lemma}

\subsection{Minimal invariant $\theta$-systems} Any root space is naturally contained in a copy of $\fsl_2$ in $\fg$. 
Here we consider the homogeneous version of this statement:
for any $[\alpha]\in [R(\fg)]$, we  consider the smallest subalgebra $\fs_{[\alpha]}$ of $\fg$ containing $\fg_{[\alpha]}$. 
It must be $\theta$-invariant and thus we get an induced grading on it. We will show that in the simplest cases, i.e. when $m=2,3$ or $5$, 
this is a simple Lie algebra  isomorphic to $\fsl_2$, $\fsl_3$ or $\fsl_5$ respectively. This will correspond to a decomposition 
of $R(\fg)$ into  $\theta$-invariant root subsystems of type $A_1$, $A_2$ or $A_4$.
\medskip

\paragraph{\bf{$m=2$.}} Here  $\alpha=\alpha_1$ and $\theta(\alpha)=-\alpha_1$. Therefore $\fs_{[\alpha]}\simeq \fsl_2$ is generated by $\fg_\alpha$, $\fg_{-\alpha}$ and $[\fg_{-\alpha},\fg_\alpha]=\CC t_\alpha$, where $t_\alpha$ is the coroot corresponding to $\alpha$. 
%In particular $\fs_\alpha\simeq \fsl_2$. %As a graded Lie algebra, $(\fs_\alpha,\theta|_{\fs_\alpha})$ corresponds to the only $\ZZ_2$-grading for which $\dim(\fg_0)=1$ and $\dim(\fg_1)=2$. By Kac's classification, this corresponds to a Cartan decomposition $\fsl_2=\fh\oplus \fsl_{2,a}\oplus \fsl_{2,-a}$ with grading given by $\fsl_{2,0}=\fh$ and $\fsl_{2,1}=\fsl_{2,a}\oplus \fsl_{2,-a}$, where we have denoted by $a$ the simple root of $\fsl_3$.
\medskip

\paragraph{\bf{$m=3$.}} In this case $\alpha=\alpha_1+\alpha_{-1}$ and $\theta(\alpha)=\xi\alpha_1+\xi^2 \alpha_{-1}$, where $\xi$ is a 3rd root of unity; $\alpha$ and $\theta(\alpha)$ have the same length $(\alpha,\alpha)=2(\alpha_1,\alpha_{-1})$. One computes that  $$\langle \theta(\alpha),\alpha \rangle:=2\frac{(\alpha,\theta(\alpha))}{(\alpha,\alpha)}=2\frac{(\xi+\xi^2)(\alpha_1,\alpha_{-1})}{2(\alpha_1,\alpha_{-1})}=-1.$$ 
Similarly  $\langle \theta^2(\alpha),\alpha \rangle = -1$. Therefore the roots $\alpha,\theta(\alpha),\theta^2(\alpha)$ generate a root system of type $A_2$; we can choose  $\alpha,\theta(\alpha)$ for simple roots, in which case $\theta^2(\alpha)=-\alpha-\theta(\alpha)$ corresponds to 
minus the highest root. 

The subalgebra $\fs_{[\alpha]}$ must contain $\fg_\alpha,\fg_{\theta(\alpha)}, \fg_{\alpha+\theta(\alpha)}, \fg_{-\alpha},\fg_{-\theta(\alpha)}, \fg_{-\alpha-\theta(\alpha)} $ (for instance $[\fg_\alpha,\fg_{\theta(\alpha)}]=\fg_{\alpha+\theta(\alpha)}$). Therefore $s_{[\alpha]}$ is 
a copy 
of  $\fsl_3$, and the direct sum of $\fg_{[\alpha]}$, $\fg_{[-\alpha]}$ and $\CC t_{\alpha} \oplus \CC t_{\theta(\alpha)}$. As a graded Lie algebra, $(\fs_{[\alpha]},\theta|_{\fs_\alpha})$ corresponds to the only $\ZZ_3$-grading of $\fsl_3$ for which $\dim(\fg_0)=2$ and $\dim(\fg_1)=\dim(\fg_{-1})=3$. 
%By Kac's classification, this corresponds to a Cartan decomposition $\fsl_3=\fh\oplus \fsl_{3,a}\oplus \fsl_{3,b} \oplus \fsl_{3,a+b}\oplus \fsl_{3,-a}\oplus \fsl_{3,-b}\oplus \fsl_{3,-a-b}$ with grading given by $\fsl_{3,0}=\fh$, $\fsl_{3,1}=\fsl_{3,a}\oplus \fsl_{3,b}\oplus \fsl_{3,-a-b}$ and $\fsl_{3,-1}=\fsl_{3,-a}\oplus \fsl_{3,-b}\oplus \fsl_{3,a+b}$, where we have denoted by $a,b$ the two simple roots of $\fsl_3$.
\medskip

%\paragraph{\bf{$m=4$:}} \vlad{in this case $m$ is not prime, but maybe by defining well the splittness condition, we can understand what happens...?}

\paragraph{\bf{$m=5$:}} Here  $\alpha=\alpha_1+\alpha_{-1}+\alpha_2 + \alpha_{-2}$ and $\theta(\alpha)=\xi\alpha_1+\xi^4 \alpha_{-1}+\xi^2\alpha_2+\xi^3 \alpha_{-2}$, where $\xi$ is a 5th root of unity; $\alpha$ and $\theta(\alpha)$ have the same length  $(\alpha,\alpha)=2(\alpha_1,\alpha_{-1})+2(\alpha_2,\alpha_{-2})$, and we need to compute $$\langle \theta(\alpha),\alpha \rangle=2\frac{(\alpha,\theta(\alpha))}{(\alpha,\alpha)}=2\frac{(\xi+\xi^4)(\alpha_1,\alpha_{-1})+(\xi^2+\xi^3)(\alpha_2,\alpha_{-2})}{2(\alpha_1,\alpha_{-1})+2(\alpha_2,\alpha_{-2})}=\langle \theta^4(\alpha),\alpha \rangle$$ and $$\langle \theta^2(\alpha),\alpha \rangle=2\frac{(\alpha,\theta^2(\alpha))}{(\alpha,\alpha)}=2\frac{(\xi^2+\xi^3)(\alpha_1,\alpha_{-1})+(\xi+\xi^4)(\alpha_2,\alpha_{-2})}{2(\alpha_1,\alpha_{-1})+2(\alpha_2,\alpha_{-2})}=\langle \theta^3(\alpha),\alpha \rangle.$$  
Since $\xi+\xi^2+\xi^3+\xi^4=-1$, we have $\langle \theta(\alpha),\alpha \rangle +  \langle \theta^2(\alpha),\alpha \rangle=-1$. But for two roots 
$\beta,\gamma$ of the same length, we always have $\langle \alpha, \beta\rangle\in\{1,0,-1\}$. So one among $\langle \theta(\alpha),\alpha \rangle$ and $\langle \theta^2(\alpha),\alpha \rangle$ must be equal to $-1$ and the other to $0$. We can and will assume that $\langle \theta(\alpha),\alpha \rangle=-1$ and $\langle \theta^2(\alpha),\alpha \rangle=0$.

Then the roots $\alpha,\theta(\alpha),\ldots,\theta^4(\alpha)$ generate a root system of type $A_4$, where  $\alpha,\theta(\alpha),\theta^2(\alpha),\theta^3(\alpha)$ can be chosen as simple roots, and $\theta^4(\alpha)=-\alpha-\theta(\alpha)-\theta^2(\alpha)-\theta^3(\alpha)$ is then minus the highest root. Thus $\fs_{[\alpha]}\simeq 
\fsl_5$ is the direct sum of $\fg_{[\alpha]}$, $\fg_{[-\alpha]}$, $\fg_{[\alpha+\theta(\alpha)]}$, $\fg_{[-\alpha-\theta(\alpha)]}$ and $\CC t_{\alpha} \oplus \CC t_{\theta(\alpha)}\oplus \CC t_{\theta^2(\alpha)}\oplus \CC t_{\theta^3(\alpha)}$. As a graded Lie algebra, $(\fs_{[\alpha]},\theta|_{\fs_\alpha})$ corresponds to the only $\ZZ_5$-grading of $\fsl_5$ for which $\dim(\fg_0)=4$ and $\dim(\fg_1)=\dim(\fg_{-1})=5$. %By Kac's classification, this corresponds to a Cartan decomposition with respect to a Cartan subalgebra $\fh$ of $\fsl_5$ such that $\fsl_{5,0}=\fh$ and $\fsl_{5,1}=\fsl_{5,a}\oplus \fsl_{5,b}\oplus \fsl_{5,c}\oplus \fsl_{5,d}\oplus \fsl_{5,-a-b-c-d}$, where we have denoted by $a,b,c,d$ the four simple roots of $\fsl_5$.

%\paragraph{\bf{$m=6$:}} \vlad{in this case $m$ is not prime, but maybe by defining well the splittness condition, we can understand what happens...?}

%\begin{remark}\label{rem_m_prime}
%We suspect that a similar situation occurs for any prime $m$, i.e. that $\fs_\alpha$ is isomorphic to $\fsl_{m}$, with $\fsl_{m,0}=\fh$ a Cartan subalgebra of $\fsl_{m}$ and $\fsl_{m,1}=\fsl_{m,a_1}\oplus \cdots \oplus \fsl_{m,a_{m-1}}\oplus \fsl_{m,-\Theta}$, where $a_1,\cdots,a_{m-1}$ are simple roots of $\fsl_{m}$ and $\Theta=\sum_j a_j$ is the highest root. \vlad{any idea? and $m$ not prime?}
%\end{remark}

\medskip
%Let us restrain ourselves to $m=2,3$ or $5$. From any root $\alpha\in R(\fg)$ we have constructed a $\theta$-invariant subalgebra $\fs_\alpha$ which is the direct sum of a subspace inside $\ft$ and of some $\fg_{[\beta]}$'s (including $\fg_{[\alpha]}$), $[\beta]\in [R(\fg)]$. 
As a result of this discussion, we can partition the root system $R(\fg)$ into a set $R(\fc)$ of equivalence classes, where $\beta$ and 
$\gamma$ are equivalent when $\fs_{[\beta]}=\fs_{[\gamma]}$. Each equivalence class is of course $\theta$-invariant.   We will refer to $R(\fc)$ as the set of $\theta$-roots of $(\fg,\theta)$; an element in $R(\fc)$ will be denoted by $\overline{\alpha}$ for $\alpha\in R(\fg)$. Notice that if $\overline{\alpha}=\overline{\beta}$ then $\alpha_1$ is proportional to $\beta_1$ inside $\fc^\vee$, which motivates the choice of the notation $R(\fc)$; in the following section this choice should become even clearer. 

\subsection{Complex reflections from the homogeneous decomposition}
\label{section_Weyl_group_graded}

Under the previous hypothesis we want to 
describe the little Weyl group $W_\theta$ of $(\fg,\theta)$. Recall that this is a complex reflection group. We will show now that  its complex reflections 
are given by reflections with respect to $\theta$-roots of $R(\fc)$. 

The following result and its proof is a straightforward generalization of a statement of  \cite{vinbergelash}.

\begin{prop}
\label{complex_reflections_from_theta_roots}
Let $\fg$ be a simple graded Lie algebra with maximal $\theta$-rank and order $m=2,3$ or $5$. For any $\overline{\alpha}\in R(\fc)$ there exists a complex reflection $w_{\overline{\alpha}}$ in $W_\theta$, of order $m$, fixing $\alpha_1^\perp \subset \fc$ and acting on $\CC \alpha_1$ by multiplication by $\xi^i$.
\end{prop}
\begin{proof}
Let $\fs_{[\alpha]}$ be the Lie subalgebra of type $\fsl_m$ associated to $\alpha$. Then $\theta|_{\fs_{[\alpha]}}$ is an order-$m$ automorphism of $\fsl_m$, thus  an inner automorphism (recall that $\fsl_m$ only has outer automorphisms of order $2$, and only for $m\ge 3$). So there exists $g\in \SL_m\subset G$ such that $\theta(\gamma)=g\gamma g^{-1}$ for any $\gamma\in \SL_m$. In particular $\theta(g)=g$, so $g\in G_0$. Clearly $g$ acts as the identity on $(\ft \cap \fs_{[\alpha]} )^\perp \subset \ft$. Thus $g$ acts as the identity on $(\ft \cap \fs_{[\alpha]} )^\perp \cap \fg_1 = \alpha_1^\perp \subset \fc$, and it acts by multiplication by $\xi^i$ on $\CC \alpha_1$. Thus the class of $g$ modulo $Z_0(\fc)$ defines a complex reflection $w_{\overline{\alpha}}$ inside $W_\theta$.
\end{proof}

In order to show that we get all the complex reflections inside $W_\theta$, we will use the explicit description of $W_\theta$ given in \cite{vinberg} for the cases we are interested in. We report in Table \ref{tab_exc_graded_vinberg} those algebras with $\theta$-$\corank=0$, $\rho_1(G_0)$ semisimple and order $m=2,3$ or $5$, with the corresponding little Weyl group. The type of $\fg$ is the type of the Dynkin diagram of $\fg$, with an additional superscript in parenthesis denoting the order of $\theta$ modulo the inner automorphisms of $\fg$. By $S^\pm_{2n}$ (respectively $S_{2n+1}$) we denoted the two spinor representations for the group $\Spin_{2n}$ (resp. $\Spin_{2n+1}$). By $\wedge^{\langle 4 \rangle}\CC^8$ we denoted the $\Sp_8$-representation associated to the fundamental weight corresponding to the long  simple  root. By $S^{\langle 2 \rangle}\CC^m$ we denoted the $\SO_m$-representation associated to twice the fundamental weight of the first root (in Bourbaki's notation). As in \cite{vinberg} a little Weyl group is denoted by its ordering in Shephard-Todd's classification in \cite{ShephardTodd} when a clearer explicit description is not available (for instance when $m=2$ $W_\theta$ is isomorphic to the Weyl group $W(T)$ of the simple Lie algebra $\fg$ of type $T$). In the last column we also included the complex reflections in the little Weyl group: by the notation $j^k$ we mean that there are $k$ reflections of order $j$, which therefore define $k/(j-1)$ fixed hyperplanes.

\begin{table}[]
\centering
\caption{Simple graded Lie algebra with $\theta$-$\corank=0$ and order $m=2,3$ or $5$ such that $\rho_1(G_0)$ is semisimple}
\label{tab_exc_graded_vinberg}
\begin{tabular}{c|c|c|c|c|c}
Type of $\fg$ & Type of $\rho_1(G_0)$  & $\fg_1$ & $m$ & $W_\theta$ & Reflections \\
\hline
$B_{2p}^{(1)}$ & $D_p\times B_p$ & $\CC^{2p}\otimes \CC^{2p+1} $ & $2$ & $W(B_{2p})$ & $2^{4p^2} $ \\
$B_{2p-1}^{(1)}$ & $D_p\times B_{p-1}$ & $\CC^{2p}\otimes \CC^{2p-1} $ & $2$ & $W(B_{2p-1})$ & $2^{4p^2-4p+1} $ \\
$D_{2p}^{(1)}$ & $D_p\times D_p$ & $\CC^{2p}\otimes \CC^{2p} $ & $2$ & $W(D_{2p})$ & $2^{4p^2-2p} $ \\
$A_{2n-2}^{(2)}$ & $B_{n-1}$ & $S^{\langle 2 \rangle}\CC^{2n-1} $ & $2$ & $W(A_{2n-2})$ & $2^{2n^2-3n+1} $ \\
$A_{2n-3}^{(2)}$ & $D_{n-1}$ & $S^{\langle 2 \rangle}\CC^{2n-2} $ & $2$ & $W(A_{2n-3})$ & $2^{2n^2-5n+3} $ \\
$D_{2p+1}^{(2)}$ & $B_p\times B_p$ & $\CC^{2p+1}\otimes \CC^{2p+1} $ & $2$ & $W(D_{2p+1})$ & $2^{4p^2+2p} $ \\
 $E_6^{(1)}$ & $A_2\times A_2 \times A_2$ & $\CC^3\otimes \CC^3 \otimes \CC^3$ & $3$ & $25$th & $3^{24} $ \\
 $E_7^{(1)}$ & $A_7$ & $\wedge^4\CC^8$ & $2$ & $W(E_7)$ & $2^{63} $ \\
 $E_8^{(1)}$ & $D_{8}$ & $S^+_{16}$ & $2$ & $W(E_8)$ & $2^{120} $ \\
 $E_8^{(1)}$ & $A_8$ & $\wedge^3 \CC^9$ & $3$ & $32$nd & $3^{80} $ \\
 $E_8^{(1)}$ & $A_4\times A_4$ & $\CC^5\otimes \wedge^2 \CC^5$ & $5$ & $16$th & $5^{48} $ \\
 $F_4^{(1)}$ & $B_3\times A_1$ & $S_7\otimes \CC^2$ & $2$ & $W(F_4)$ & $2^{24} $ \\
 $F_4^{(1)}$ & $A_2\times A_2$ & $\Sym^2(\CC^3)\otimes \CC^3$ & $3$ & $5$th & $3^{16} $ \\
 $G_2^{(1)}$ & $A_1\times A_1$ & $\Sym^3(\CC^2)\otimes \CC^2$ & $2$ & $W(G_2)$ & $2^{6} $ \\
 $E_6^{(2)}$ & $C_4$ & $\wedge^{\langle 4 \rangle}\CC^8$ & $2$ & $W(E_6)$ & $2^{36} $ \\
 $D_4^{(3)}$ & $A_2$ & $\Sym^3(\CC^3)$ & $3$ & $4$th & $3^{8} $ \\
 \end{tabular}
 \end{table}

\begin{prop}
\label{prop_refl}
Let $\fg$ be a simple graded Lie algebra with maximal $\theta$-rank and order $m=2,3$ or $5$ such that $\rho_1(G_0)$ is semisimple. Then the complex reflections inside $W_\theta$ are those described in Proposition \ref{complex_reflections_from_theta_roots}. 
\end{prop}
\begin{proof} 
We just check that we get the correct number of complex reflections from Proposition \ref{complex_reflections_from_theta_roots}. 
In order to determine the cardinality of $R(\fc)$, we proceed as follows. 
 For any root $\alpha\in R(\fg)$, decompose the corresponding coroot $h_\alpha\in\ft$ as $h_\alpha=\sum_i h_{\alpha,i}$, $h_{\alpha,i}\in \ft_i$.
 Let  $\fh_\alpha$ denote the Lie algebra associated to the torus $\overline{\exp(\CC h_{\alpha,1})}$. This  is a Cartan subalgebra of $\fs_{[\alpha]}$, and $\fh_\alpha\cap \fg_1 = \CC \alpha_1$. Therefore if $m=2, 3, 5$ one can find inside $\fh_\alpha$ a root subsystem 
 of $R(\fg)$ of type $A_1, A_2, A_4$. Each of the roots in such a subsystem is a representative of $\overline{\alpha}\in R(\fc)$. If we exclude the case $G_2$ with $m=3$ (because $G_2$ contains two subsystems of type $A_2$, with roots of different lengths,  generating the same Cartan subalgebra), 
 we see that any element $\overline{\alpha}$ inside $R(\fc)$ defines a line  $\CC\alpha_1\subset \fc$, and that $\CC\alpha_1 = \CC \beta_1$ if and only if $\fh_\alpha=\fh_\beta$, or equivalently $\overline{\alpha}=\overline{\beta}$. 
 
 Therefore, (except for $G_2$ with $m=3$, a case that can be treated separately), we can compute the number of distinct reflections constructed 
 in Proposition \ref{complex_reflections_from_theta_roots}. They are all of order $m$, and for each element in $R(\fc)$ we obtain $m-1$ distinct reflections whose fixed hyperplane is the same. The cardinality of $R(\fc)$ is therefore equal to $|R(\fg)|/m(m-1)$, and we check from 
 \cite{ShephardTodd} that this 
 matches the number  of hyperplanes fixed by reflexions inside $W_\theta$ for all the cases  appearing in Table \ref{tab_exc_graded_vinberg}. 
\end{proof}

\begin{remark}
Notice that when $m=2$ we are looking at $A_1$-subsystems, and thus it is not surprising that the little Weyl group in this case is exactly equal to the Weyl group of the ambient Lie algebra.
\end{remark}

\section{Projective duality}
\label{sec_5}

We suppose again in this section that $\fg$ is an exceptional simple graded Lie algebra with maximal $\theta$-rank and order $m=2,3$ or $5$, 
such that $\rho_1(G_0)$ is semisimple. Since $\fg_1$ is irreducible it contains a minimal (non trivial) nilpotent orbit.  Start with a minimal orbit inside $\fs_{[\alpha]}\cong \fsl_m$ for some $\overline{\alpha}\in R(\fc)$. Recall that  $\fs_{[\alpha],0}=\fs_{[\alpha]} \cap \fg_0$ is a Cartan subalgebra of $\fs_{[\alpha]}$. Let $$\fs_{[\alpha]}=\fs_{[\alpha],0}\oplus \bigoplus_{a\in R(A_{m-1})}\fs_{[\alpha],a}$$ be the corresponding root decomposition. From the discussion in Section \ref{sec_homogeneous_decomposition}, we know that $\fs_{[\alpha],1}$ is the direct sum of the root spaces $\fs_{[\alpha],a_i}$, $i=1,\ldots,m-1$, for $a_1,\ldots, a_{m-1}$ a  basis of simple roots in $R(A_{m-1})$, and 
of the root space $\fs_{\alpha,-\sum_i a_i}$ corresponding to minus the highest root in $R(A_{m-1})$. We will study  the $G_0$-orbit of a nilpotent element $0\neq e\in \fs_{[\alpha],a_i}$ for a certain $i$.

\subsection{An invariant hypersurface}
\label{sec_degree_D}

For $\overline{\alpha}\in R(\fc)$, consider $\alpha_{-1}=\alpha_{|\fg_{-1}}\in \ft_{-1}^\vee$. 
The polynomial 
\[
D:=\prod_{u\in W_\theta(\alpha_{-1})}u
\] is a $W_\theta$-invariant homogeneous polynomial in $\CC[\fc_{-1}]$. By Vinberg's extension of the Chevalley restriction theorem, 
this corresponds to a $G_0$-invariant polynomial in $\CC[\fg_{-1}]$, which we will still denote by $D$, hence  to an invariant  
hypersurface $\cD=\{ D=0 \}\subset \PP(\fg_{-1})$.

The inclusion $\CC[\ft_{-1}]^{W_\theta}\subset \CC[\ft_{-1}]$ corresponds to the  projection $\ft_{-1}\to \ft_{-1}/W_\theta$, which sends 
the hyperplane $\{\alpha_{-1}=0\}$ to an irreducible hypersurface. So $\cD$ as well must be an irreducible hypersurface in $\PP(\fg_{-1})$. 

\medskip

\paragraph{\bf{Conjugacy classes of reflections.}} The definition of $D$ only depends on the $W_\theta$-conjugacy class of the hyperplane $\alpha_{-1}^\perp$. As shown in Proposition \ref{complex_reflections_from_theta_roots} this hyperplane is the invariant hyperplane of the complex reflection $w_{\overline{\alpha}}$. In the case of Weyl groups (and thus when $m=2$), the situation is well known: if the Dynkin diagram of the corresponding root system is simply laced, there is only one conjugacy class of reflections and hyperplanes, each element of the class corresponding to a positive root of the root system; if the Dynkin diagram of the corresponding root system is not simply laced, there are two conjugacy classes of reflections and hyperplanes, corresponding to positive long roots or  positive short roots.

In general, for exceptional complex reflection groups, there can be at most three conjugacy classes of hyperplanes and reflections. In \cite{bcm} it is shown for instance that the $4$th, $16$th, $25$th and $32$nd group in Shephard-Todd's list have only one conjugacy class of reflections, while the $5$th group in Shephard-Todd's list has two. Apart from the $5$th group case, since we know the total number of reflections in each complex reflection group (and the total number of long or short roots when $m=2$ and the Dynkin diagram is not simply laced), we can deduce the degree of $D$ in all the cases appearing in Table \ref{tab_exc_graded_vinberg}. Indeed, the degree of $D$ will be equal to 
\[
\deg(D)=m \mbox{  }| \{w_{g(\overline{\alpha})}\mid g\in W_\theta\}|,
\]
where  $| \{w_{g(\overline{\alpha})}\mid g\in W_\theta\}|$ is the cardinality of the set of reflections conjugate  to $w_{\overline{\alpha}}$.
For the $5$th group case,  the two conjugacy classes of reflections correspond to the two types of $A_2$-root subsystems  (which correspond to $\theta$-subsystems and hence reflections, since in this case $m=3$) formed by long and short roots of $F_4$; there is an equal number of short and long roots in $F_4$ and there is an equal number of long and short $A_2$-root subsystems. 

\smallskip
In Table \ref{tab_degree_D} we reported the degrees of the polynomial $D$ when the graded Lie algebra is among those appearing in Table \ref{tab_exc_graded_vinberg}. For the non simply laced case, we denoted by $\alpha_l$ the conjugacy class of long roots and by $\alpha_s$ the conjugacy class of short roots; for the $5$th group, we denoted similarly the two conjugacy classes of reflections by $\alpha_l$ and $\alpha_s$.

\begin{table}[]
\centering 
\label{tab_degree_D}
\small
\begin{tabular}{c|c|c|c|c|c}
$W_\theta$, root & $W(A_{2n-2})$, $\alpha $ & $W(B_{2p})$, $\alpha_l $
& $W(B_{2p})$, $\alpha_s $ &  $W(E_6)$, $\alpha $  & $32$nd, $\alpha $ \\
\hline
 $\deg(D)$ & $(2n-1)(2n-2)$ &  $4p(2p-1)$ & $4p$ &  $ 72$  & $120$   \\
  \\
$W_\theta$, root & $W(A_{2n-3})$, $\alpha $ & $W(D_{2p+1})$, $\alpha$ & $W(B_{2p-1})$, $\alpha_s $ &  
$W(E_7)$, $\alpha $    &  $16$th, $\alpha $ \\
\hline
 $\deg(D)$ &  $(2n-2)(2n-3)$ & $ 4p(2p+1)$ &  $4p-2$ &  $126$ & $60$  \\
  \\
  $W_\theta$, root & $W(B_{2p-1})$, $\alpha_l $ &  $W(D_{2p})$, $\alpha $ & $25$th, $\alpha $&  $W(E_8)$, $\alpha $&
 $4$th, $\alpha $     \\
 \hline 
 $\deg(D)$ & $4(p-1)(2p-1)$&  $4p(2p-1)$ & $36$ &  $240$ & $12$ \\
   \\
 $W_\theta$, root & $W(G_2)$, 
 $\alpha_l \;\mathrm{or}\;\alpha_s$  & $5$th, $\alpha_l \;\mathrm{or}\;\alpha_s$ &  $W(F_4)$, $\alpha_l \;\mathrm{or}\;\alpha_s$ &   
 & \\
 \hline 
 $\deg(D)$ & $6$ & $ 12$ & $24$ &  & \\
\\
 \end{tabular}

\caption{Degree of the polynomial $D$}
 \end{table}

\subsection{The main statement}
\label{sec_main_theorem}

The strategy of the proof (which is essentially the same as the one used by Tevelev to prove Theorem \ref{tevelev}) is to compute the set of hyperplanes containing the tangent space at $e\in \fg_1$ of the orbit $G_0 e$; then, via the action of $G_0$, one recovers the dual variety of the closure of $G_0 [e]\subset  \PP(\fg_1)$. 

\begin{prop}{\cite[Proposition 5]{vinberg}}
The tangent space to the orbit of an element $e\in \fg_1$ is the orthogonal complement to the centralizer of $e$ in $\fg_{-1}$. In other words 
\[
T_{G_0 e , e}^\perp = \fz_{-1}(e):=\{ x\in \fg_{-1}\mid [x,e]=0 \}.
\]
\end{prop}

%\paragraph{\bf{Intersection of centralizer and Cartan subspace:}} Let $e\in \fs_\alpha \cap \fg_1$ be a nilpotent element such that $\alpha_1(e)\neq 0$. Then the intersection $\fz_{-1}(e)\cap \ft_{-1}$ is equal to the Killing-orthogonal complement $\alpha_1^\perp\subset \ft_{-1}$. 
%\medskip

%\paragraph{\bf{A general nilpotent element in the centralizer:}} 
%\medskip

The following theorem is our generalization of Tevelev's formula.

\begin{theorem}
\label{thm_main}
Let $\fg$ be a simple graded Lie algebra with maximal $\theta$-rank and order $m=2,3$ or $5$. Let $\overline{\alpha}\in R(\fc)$, where $\fc$ is a Cartan subspace of $\fg$, and $\fs_{[\alpha]}\subset \fg$ the corresponding $\theta$-subsystem. Let $e$ be a nilpotent element inside $\fs_{[\alpha]}\cap \fg_1$ contained in a root space for the Cartan subalgebra $\fs_{[\alpha]}\cap \fg_0$ of the simple Lie algebra $\fs_{[\alpha]}$. Then 
\[
(\overline{G_0[e]})^\vee=\cD\subset \PP(\fg_1)^\vee.
\]
\end{theorem}

\begin{remark}
A posteriori $G_0[e]$ does not depend on the choice of $e$, but only on the choice of $\alpha$. Indeed $\cD$ only depends on the $W_\theta$-orbit of $\alpha_{-1}$, and $\cD^\vee=((\overline{G_0[e]})^\vee)^\vee=\overline{G_0[e]}$.
\end{remark}

\begin{remark}
\label{rmk_two_orbits}
Recall that Tevelev's result actuallyholds for two orbit closures (in the non simply laced case) inside the adjoint representation. The first one is the (closed) orbit of the root spaces associated to long roots; the second one, in the non simply laced case, is the orbit closure of the root spaces associated to short roots. Similarly, Theorem \ref{thm_main} holds in general for two orbit closures: the one associated to $\alpha_l$ and the one associated to $\alpha_s$ (at least when there are two conjugacy classes of reflections in the little Weyl group; otherwise there is only one conjugacy class $\alpha=\alpha_l$, see Table \ref{tab_degree_D}). As we will show in Proposition \ref{prop_closed_orbit}, the orbit closure associated to $\alpha_l$ gives the closed orbit in $\fg_1$. 
\end{remark}

In view of Remark \ref{rmk_two_orbits}, we will denote by $\cD_l$ (respectively $\cD_s$) the dual of the $G_0$-orbit closure $\overline{G_0[e]}$ with $e\in \fs_{[\alpha_l]}$ (resp. $e\in \fs_{[\alpha_s]}$). As a consequence we can refine \cite[Theorem 3.1]{HO}. 

\begin{coro} 
Let $\cD_{\fg,l}\subset\PP(\fg)$ (respectively $\cD_{\fg,s}\subset\PP(\fg)$) denote the discriminant of $\fg$ (resp. the dual of the $G$-orbit closure of vectors in short root spaces). The intersection of $\cD_{\fg,l}$ (resp. $\cD_{\fg,s}$) with $\PP(\fg_{-1})$ is $(m-1)\cD_l$ (resp. $(m-1)\cD_s$). 
\end{coro} 

\proof 
Let $\epsilon\in\{l,s\}$. From Tevelev's formula we know that $D_{\fg,\epsilon} = \prod_{u\in W(\alpha_\epsilon)}u$, while $D_{\epsilon} = \prod_{u\in W_\theta((\alpha_\epsilon)_{-1})}u$. Therefore $D_{\fg,\epsilon}|_{\fc_{-1}} = \prod_{u\in W(\alpha_\epsilon)} u_{-1}$.

By \cite[Remark after Proposition 8]{vinberg}, the little Weyl group $W_\theta$ is equal to $W^\theta$ since the $\theta$-rank
is maximal. Thus $W_\theta$ is contained in the Weyl group $W$, and any element $u\in W_\theta((\alpha_\epsilon)_{-1})$ is equal to $\beta_{-1}$ for a certain root $\beta$. Moreover, all restrictions $\beta_{-1}$ for all roots $\beta$ of fixed length are in the same $W_\theta$-orbit since the corresponding reflections are all conjugate  in $W_\theta$ (see Section \ref{sec_degree_D}). This implies that 
$$
D_{\fg,\epsilon}|_{\fc_{-1}} = \prod_{u\in W_\theta((\alpha_\epsilon)_{-1})} \Large(\prod_{\substack{\beta \in W(\alpha_\epsilon)\\ \beta|_{\fc_{-1}}=u}}u\Large)=\prod_{u\in  W_\theta((\alpha_\epsilon)_{-1})} u^{|\{  \beta \in W(\alpha_\epsilon)\mid \,\, \beta|_{\fc_{-1}}=u \}|}.
$$
In order to conclude we need to show that, for each $u\in W_\theta((\alpha_\epsilon)_{-1})$, the cardinality $|\{  \beta \in W(\alpha_\epsilon)\mid \,\, \beta|_{\fc_{-1}}=u \}|$ is at least $m-1$; since $\deg(D_{\fg,\epsilon})=(m-1)\deg(D_\epsilon)$ by an explicit check, this cardinality is constant and equal to $m-1$, and the result follows. So, let us fix $u=\beta_{-1}$ for a certain root $\beta$ of length $\epsilon$. Then there are $m(m-1)$ roots in $\fs_{[\beta]}\cong \fsl_{m}$. These roots are partitioned in $m$ subsets of $m-1$ roots each, each subset being the set of roots $\gamma$ such that $\gamma|_{\fc_{-1}}=\theta^i(u)=\xi^i u$, for $i=0,\cdots,m-1$. So there are $m-1$ such roots $\gamma$ such that $\gamma|_{\fc_{-1}}=u$.\qed 

\medskip 

\begin{proof}[Proof of Theorem \ref{thm_main}]
The dual projective variety of $\overline{G_0[e]}\subset \PP(\fg_{1})$ is the closure 
\[
\overline{G_0 \PP(T_{G_0 e,e}^\perp)}=\overline{G_0 \PP(\fz_{-1}(e))}\subset \PP(\fg_{-1})=\PP(\fg_1)^\vee, 
\]
where we have used the fact that $\fg_{-1}$ and $\fg_1$ are dual under the Killing form. We will first show that $\PP(\fz_{-1}(e))\subset \cD$, then that $G_0 \PP(\fz_{-1}(e)))$ is a hypersurface in $\PP(\fg_{-1})$. Since $\cD$ is irreducible, the statement will follow.\medskip

\paragraph{\bf{Inclusion inside the invariant divisor.}} We want to show that, for any element $x\in\fz_{-1}(e)$, $D(x)=0$. Recall that for such an element $x$, if $x=x_s+x_n$ is its homogeneous Jordan decomposition in its semisimple part $x_s\in\fg_{-1}$ and its nilpotent part $x_n\in\fg_{-1}$, then $[x,e]=[x_s,e]=[x_n,e]=0$. Thus $D(x)=D(x_s)$, where $x_s\in \fz_{-1}(e)$ is semisimple. We are thus reduced to the case 
where $x\in\fz_{-1}(e)$ is semisimple. 

Clearly $x\in \ft_{-1}\cap \fz_{-1}(e)$ if and only if $x\in \Ker(\alpha)$, i.e. if and only if $\alpha(x)=\alpha_{-1}(x)=0$. Indeed, as already noticed in the proof of Proposition \ref{prop_refl}, each root $\beta$ in $\fs_{[\alpha]}$ satisfies $\beta_1=\alpha_1$, and in particular this holds for the root $\beta$ whose root space the vector $e$ belongs to. Thus the hyperplane $\alpha_{-1}^\perp\subset \ft_{-1}$ is equal to $\fz_{-1}(e)\cap \ft_{-1}$.

 In general, since $x$ is semisimple, $(\fz(e),\theta|_{\fz(e)})$ is a graded reductive Lie algebra; by Vinberg's theory, each Cartan subspace in $\fz_{-1}(e)$ is in the same conjugacy class under $Z_0(e):=\{g\in G_0\mid g(e)=e \}$. Let us check that $ \alpha_{-1}^\perp$ is a Cartan subspace of $\fz_{-1}(e)$. If this were not the case, there would exist $g\in Z_0(e)$ such that $\CC g(x) \oplus \alpha_{-1}^\perp$ is a Cartan subspace of $\fg$ inside $\fz_{-1}(e)$; then the Lie algebra $\tilde{\ft}$ associated to $\CC g(x) \oplus \alpha_{-1}^\perp$ would be a Cartan subalgebra of $\fg$ satisfying $[\tilde{\ft},e]=0$, which is a contradiction with the fact that $e$ is nilpotent. 
 
Since, as already remarked, all Cartan subspaces are conjugate  and $\alpha_{-1}^\perp$ is a Cartan subspace of $\fz_{-1}(e)$, there exists $g\in Z_0(e)$ such that $g(x)\in \alpha_{-1}^\perp$. Then $D(g(x))=0$, and by the $G_0$-invariance of $D$ we get $D(x)=0$. Therefore we can conclude that $\PP(\fz_{-1}(e))\subset \cD$.
%In general, since $x$ is semisimple and all Cartan subspaces are conjugate  under $G_0$, there exists $g\in G_0$ such that $g(x)\in \ft_{-1}$. The element $g(e)$ is nilpotent and $g(x)\in \fz_{-1}(g(e))=g(\fz_{-1}(e))$, therefore there exists a root $\beta$ such that $\beta(g(x))=\beta_{-1}(g(x))=0$.\vlad{$\beta_1\in G_0 \alpha_1?$} This implies that $D(g(x))=0$, and since $D$ is $G_0$-invariant, that $D(x)=0$. 

\medskip
\paragraph{\bf{Dimension count.}} Now we compute  the dimension of $\PP(G_0(\fz_{-1}(e)))$. Since it is contained in $\cD$, in order to show 
that it is a divisor it is sufficient to show that 
$$\dim(G_0(\fz_{-1}(e)))\ge \dim(\fg_{-1})-1.$$
First notice that $\dim(\fg_{i})= \dim(\fg_{0})+\dim(\fc)$ for $i\neq 0$. Indeed, recall the homogeneous decomposition $\fg=\ft \oplus \bigoplus_{[\alpha]\in[R(\fg)]}\fg_{[\alpha]}$. On one hand $\ft=\bigoplus_{i\neq 0}\ft_i$
 with $\ft_i=\ft\cap \fg_i$ and $\dim(\ft_i)=\dim(\ft_1)=\dim(\fc)$. On the other hand, by Lemma \ref{lem_homogeneous_component} we know that $\fg_{[\alpha]}=\bigoplus_i (\fg_{[\alpha]}\cap \fg_i)$ with $\dim(\fg_{[\alpha]}\cap \fg_i)=1$. Hence the claim. Note moreover that the dimension of $\fg_0$ is equal to $|[R(\fg)]|=|R(\fg)|/m$.
 
We will bound $\dim(G_0(\fz_{-1}(e)))$ from below by computing $\dim(G_0(\hat{\fz}))$ for a certain subset $\hat{\fz}\subset \fz_{-1}(e)$. Consider the subalgebra $\fs_{[\alpha]}$, and $e\in \fs_{[\alpha]}\cap \fg_1$ a root vector. 

\begin{lemma}
$\fz_{-1}(e)\cap \fs_{[\alpha]}$ contains a regular nilpotent element $f$ of $\fs_{[\alpha]}$.
\end{lemma}

\proof Recall that a nilpotent element is regular if it belongs to the dense orbit of  the nilpotent cone. A typical example 
is a linear combination of the simple root vectors, with nonzero coefficients. 

In our situation, $\fs_{[\alpha]}\cong\fsl_m$, the subspace $\fs_{[\alpha]}\cap \fg_0$ is a Cartan subalgebra, and $\fs_{[\alpha]}\cap \fg_1=\bigoplus_{i=1}^{m-1}\fsl_{m,a_i} \oplus \fsl_{m,-\sum_i a_i}$ for $a_1,\ldots,a_{m-1}$ a simple root basis of $\fsl_m$. 

We can choose $e$ to be a root vector corresponding to minus the highest weight of $\fs_\alpha$. Indeed, we chose $e$ to belong to a root space with respect to $\fs_{[\alpha]}\cap \fg_0$, thus $e$ belongs to one of the $\fsl_{m,a_i}$'s or to $\fsl_{m,-\sum_i a_i}$; if for instance $e$ belongs to $\fsl_{m,a_j}$, then notice that $a_j$ is minus the highest root with respect to the root basis $a_1,\cdots,\hat{a_j},\ldots,a_{m-1},-\sum_i a_i$. 

In other words, we are choosing $0\neq e\in \fsl_{m,-\sum_i a_i}$, and then the subspace $\bigoplus_{i=1}^{m-1}\fsl_{m,-a_i}\subset \fs_{[\alpha]}\cap \fz_{-1}(e)$ obviously contains a regular nilpotent element.  \qed 

\medskip Such an $f$ being chosen, consider the affine space
\[
\hat{\fz}:=\{ x= f + y\mid y\in \alpha_{-1}^\perp \subset \fc \}.
\]
Clearly $\hat{\fz}\subset \fz_{-1}(e)$, $\dim(\hat{\fz})=\dim(\alpha_{-1}^\perp)=\dim(\fc)-1$ and $x=f+y$ is the Jordan decomposition of $x$ (with $y$ semisimple and $f$ nilpotent; notice that $f$ and $y$ commute since $f\in \fg_{[\alpha]}$). Then 
\[
\dim(G_0(\hat{\fz}))
=\dim(\fg_0)+\dim(\hat{\fz})-\dim(N_0(\hat{\fz}))
=\dim(\fg_{-1})-1-\dim(N_0(\hat{\fz})),
\]
where  $N_0(\hat{\fz}):= \{g\in G_0\mid g(\hat{\fz})=\hat{\fz} \}$ has Lie algebra $\fn_0(\hat{\fz}):= \{g\in \fg_0\mid [g,\hat{\fz}]\subset \hat{\fz} \}$. So in order to prove the theorem there remains to show that $\fn_0(\hat{\fz})=0$. 

Consider $g\in N_0(\hat{\fz})$. Since $g$ must preserve the Jordan decomposition of any $x=f+y\in \hat{\fz}$, we need $g(f)=f$ and $g(y)\in \alpha_{-1}^\perp$. In other words $N_0(\hat{\fz})$ is the centralizer of $f$ inside the normalizer $N_0(\alpha^\perp_{-1})$. The Lie algebra of this normalizer is just $\fn(\alpha_{-1}^\perp)\cap \fg_0= (\ft+\fs_{[\alpha]})\cap \fg_0=\fs_{[\alpha]}\cap \fg_0$, where $\fn(\alpha_{-1}^\perp):=\{ g\in \fg\mid [g,\alpha_{-1}^\perp]\subset \alpha_{-1}^\perp  \}$. However $\fs_{[\alpha]}\cap \fg_0$ is a Cartan subalgebra of $\fs_{[\alpha]}\cong \fsl_m$. Since $f$ is regular nilpotent inside $\fs_{[\alpha]}$, it is not centralized by any semisimple element and  we can 
finally conclude that $\fn_0(\hat{\fz})=0$.
\end{proof}

\section{Explicit examples}
\label{sec_6}

\subsection{Closed orbits}

We will describe the varieties to which the theorem applies for the graded Lie algebras appearing in Table \ref{tab_exc_graded_vinberg}. 
So $\fg$ will be a simple graded Lie algebra with $\theta$-$\corank=0$ and order $m=2,3$ or $5$, such that $\rho_1(G_0)$ is semisimple. Notice however that the hypothesis \emph{$\rho_1(G_0)$ semisimple} is not strictly necessary in order to apply the theorem.
\medskip

\paragraph{\bf{List of closed orbits.}} In all the cases, the $G_0$-representation $\fg_1$ is irreducible and there is a single closed (or minimal) orbit inside $\PP(\fg_1)$. In Table \ref{tab_orbits} we report the list of these closed orbits. We denoted by $\OG$ the orthogonal Grassmannian of isotropic subspaces with respect to a non-degenerate symmetric $2$-form (if the subspaces have maximal dimension and the   ambient vector space is even dimensional, a subscript $\pm$ will denote one of the two connected components of the family of such subspaces). Moreover $v_i$  denotes the $i$-th Veronese embedding.

\begin{table}[]
\centering
\caption{Closed $G_0$-orbits inside $\PP(\fg_1)$ for the graded Lie algebras of Table \ref{tab_exc_graded_vinberg} and their codegrees}
\label{tab_orbits}
\begin{tabular}{c|c|c|c|c|c}
Type of $\fg$ & Type of $\rho_1(G_0)$  & $\fg_1$ & Orbit & $\deg(\cD)$ \\
\hline
$B_{2p}^{(1)}$ & $D_p\times B_p$ & $\CC^{2p}\otimes \CC^{2p+1} $  & $\QQ^{2p-2}\times \QQ^{2p-1}$ & $4p(2p-1)$ \\
$B_{2p-1}^{(1)}$ & $D_p\times B_{p-1}$ & $\CC^{2p}\otimes \CC^{2p-1} $  & $\QQ^{2p-2}\times \QQ^{2p-3}$ & $4(p-1)(2p-1)$ \\
$D_{2p}^{(1)}$ & $D_p\times D_p$ & $\CC^{2p}\otimes \CC^{2p} $  & $\QQ^{2p-2}\times \QQ^{2p-2}$ & $4p(2p-1)$ \\
$A_{2n-2}^{(2)}$ & $B_{n-1}$ & $S^{\langle 2 \rangle}\CC^{2n-1} $  & $v_2(\QQ^{2n-3})$ & $2(n-1)(2n-1)$ \\
$A_{2n-3}^{(2)}$ & $D_{n-1}$ & $S^{\langle 2 \rangle}\CC^{2n-2} $  & $v_2(\QQ^{2n-4})$ & $2(n-1)(2n-3)$ \\
$D_{2p+1}^{(2)}$ & $B_p\times B_p$ & $\CC^{2p+1}\otimes \CC^{2p+1} $  & $\QQ^{2p-1}\times \QQ^{2p-1}$ & $4p(2p+1)$ \\
 $E_6^{(1)}$ & $A_2\times A_2 \times A_2$ & $\CC^3\otimes \CC^3 \otimes \CC^3$ & $\PP^2\times \PP^2\times \PP^2$ & $36$ \\
 $E_7^{(1)}$ & $A_7$ & $\wedge^4\CC^8$ & $G(4,8)$ & $126$  \\
 $E_8^{(1)}$ & $D_{8}$ & $S^+_{16}$ & $\OG(8,16)_+$ &  $240$ \\
 $E_8^{(1)}$ & $A_8$ & $\wedge^3 \CC^9$ & $G(3,9)$ & $120$  \\
 $E_8^{(1)}$ & $A_4\times A_4$ & $\CC^5\otimes \wedge^2 \CC^5$ & $\PP^4\times G(2,5)$ & $60$ \\
 $F_4^{(1)}$ & $B_3\times A_1$ & $S_7\otimes \CC^2$ & $\OG(3,7)\times \PP^1$ & $24$ \\
 $F_4^{(1)}$ & $A_2\times A_2$ & $\Sym^2(\CC^3)\otimes \CC^3$ & $v_2(\PP^2)\times \PP^2$ & $12$ \\
 $G_2^{(1)}$ & $A_1\times A_1$ & $\Sym^3(\CC^2)\otimes \CC^2$ & $v_3(\PP^1)\times \PP^1$ & $6$ \\
 $E_6^{(2)}$ & $C_4$ & $\wedge^{\langle 4 \rangle}\CC^8$ & $LG(4,8)$ & $72$ \\
 $D_4^{(3)}$ & $A_2$ & $\Sym^3(\CC^3)$ & $v_3(\PP^2)$ & $12$ \\
 \end{tabular}
 \end{table}
\medskip

\paragraph{\bf{Orbits description.}} In every $G_0$-representation of Table \ref{tab_exc_graded_vinberg} there is a unique closed orbit. 
We want now to identify it with one of the orbits whose dual is described in Theorem \ref{thm_main} (see also Remark \ref{rmk_two_orbits}). We will use an argument 
already present in the proof of \cite[Proposition 2]{vinberg}. We keep the notations of Section \ref{sec_5}.

\begin{prop}
\label{prop_closed_orbit}
Let $\fg$ be an exceptional simple graded Lie algebra with maximal $\theta$-rank and order $m=2,3$ or $5$, such that $\rho_1(G_0)$ is semisimple. %Let $\overline{\alpha}\in R(\fc)$, where $\fc$ is a Cartan subspace of $\fg$, and $\fs_\alpha\subset \fg$ the corresponding $\theta$-subsystem.
Suppose that $\alpha$ is a long root. Let $e$ be a nonzero nilpotent element inside $\fs_{[\alpha]}\cap \fg_1$, contained in a root space for the Cartan subalgebra $\fs_{[\alpha]}\cap \fg_0$ of the simple Lie algebra $\fs_{[\alpha]}$. Then $G_0[e]\subset \PP(\fg_1)$ is the unique closed  $G_0$-orbit inside $\PP(\fg_1)$.
\end{prop}

\begin{proof}

It is clear that for any element $e\in\fg_1$, $[\fg,e]\cap \fg_1=[\fg_0,e]$. Let $v\in G e \cap \fg_1$. The tangent space at $v$ satisfies $T_{G e \cap \fg_1,v}\subset [\fg,v]\cap \fg_1=[\fg_0,v]=T_{G_0 v,v}$. Since $G_0 v\subset G v \cap \fg_1=G e \cap \fg_1$, we deduce that in fact 
$T_{G_0 v,v}= T_{G e \cap \fg_1,v}$. Since this holds for any point $v\in G e \cap \fg_1$, we conclude that $G e \cap \fg_1$ is smooth 
and that each irreducible component of $G e \cap \fg_1$ is a $G_0$-orbit. 

Since $e$ belongs to the root space corresponding to a long root of $R(\fg)$, the orbit $G[e]\subset \PP(\fg)$ is the closed orbit of $\fg$,
namely the adjoint variety. As a consequence, the intersection 
$G[e]\cap \PP(\fg_1)$ is closed inside $\PP(\fg_1)$, and each of its irreducible components 
is also closed. But we have just seen that each of these components is a $G_0$-orbit, and since there is a unique such closed $G_0$-orbit
inside $\PP(\fg_1)$, it has to coincide with $G[e]\cap \PP(\fg_1)$, and with $G_0[e]$ as well. 
\end{proof}

In Table \ref{tab_orbits} we have reported the degree of $\cD$ corresponding to the choice of a long root (recall from Section \ref{sec_degree_D} that under the hypothesis we have made, we have defined at most two invariant divisors for each graded Lie algebra, each one corresponding to long or - possibly - short roots); by the previous Proposition and Theorem \ref{thm_main}, $\cD$ is the dual variety of the closed orbit in $\PP(\fg_1)$.

%\paragraph{\bf{$m=2$:}}  \medskip \paragraph{\bf{$E_6^{1}, m=3$:}}  \medskip \paragraph{\bf{$E_8^{1}, m=3$:}}  \medskip \paragraph{\bf{$E_8^{1}, m=5$:}}  \medskip \paragraph{\bf{$F_4^{1}, m=3$:}}  \medskip \paragraph{\bf{$D_4^{3}, m=3$:}} 
\medskip

\subsection{Classical cases} 

The classification of classical graded simple Lie algebras involves the classical affine Dynkin diagrams plus the Kac 
diagrams of type $A^{(2)}_{n}$ and $D^{(2)}_{n}$. The coefficients of these diagrams are either $m=2$ or $m=4$. We have 
excluded the latter case from our considerations. So under the hypothesis that $\rho_1(G_0)$ is semisimple 
there only remains $\ZZ_ 2$-gradings of the following type:
\[
\fsl_{2n}=\fsp_{2n}\oplus \wedge^{\langle 2\rangle}\CC^{2n}, \qquad 
\fsl_{2n}=\fso_{2n}\oplus S^{\langle 2\rangle}\CC^{2n}
\]
and, coming either from $B^{(1)}_{n}$, $D^{(1)}_{n}$ or $D^{(2)}_{n}$ according to the parity of 
$a$ and $b$, 
\[
\fso_{a+b}=\fso_a\times\fso_b\oplus (\CC^a\otimes\CC^b).
\]
%Applying the main theorem we get the following codegrees:
%\begin{coro}$\deg(IG(2,2n)^\vee)=\deg (v_2(\QQ^{n-2})^\vee=n(n-1). $ \end{coro}
In the latter case, it is shown in \cite{helgason} that the associated $\theta$-rank is equal to $\min(a,b)$. We have seen that the main theorem applies when $a=b$ or $a=b\pm 1$ (for such $a,b$ one can check that $\theta$-$\corank$=0), giving as codegree of $\QQ^{a-2}\times \QQ^{b-2}$ the integer $4{\min(a,b)\choose 2}$. We expect that a similar statement should hold for any $a,b$, that is
$$\deg\Big( (\QQ^{a-2}\times\QQ^{b-2})^\vee\Big)=4\binom{\min(a,b)}{2}.$$
One should be able to prove this statement by some different argument. 
Notice that by the Katz-Kleiman formula, this codegree is equal to four times the coefficient of $x^ay^b$ in the Taylor expansion of 
$$\frac{(1-x)^{a+2}(1-y)^{b+2}}{(1-2x)(1-x-y)^2(1-2y)}.$$
%Even though it may apparently seem otherwise, this coefficient is not easily computed. 

%\laurent{Expliquer comment cela se déduit du Theoreme} \vlad{Donc, ce que j'ai compris: dans ces cas, $m=2$ et $\min(a,b)=r=\theta$-$\rank$; il semblerait que Vinberg dit que dans ce cas $W_\theta=G(m,1,r)$ ou $W_\theta=G(m,2,r)$ (on est dans le "second case, type I or type II"). Je serais tenté de dire que $W_\theta=G(m,2,r)$ quand $a+b$ est pair, c-a-d quand il s'agit de $D_n$, alors que $W_\theta=G(m,1,r)$ si $a+b$ est impair, c-a-d quand il s'agit de $B_n$ (tu confirmes?). Dans $G(m,2,r)$ il y a que des reflections d'ordre $2$, et il y en a $4{r\choose 2}$. D'autre part, dans $G(m,1,r)$ il y a $4{r\choose 2}$ reflections d'ordre $2$ (ce qu'on veut) plus encore $r$ reflections d'ordre $2$. Probablement ce qui se passe dans ce deuxième cas est que ces dernières $r$ reflections viennent des racines courtes et ne doivent pas être prises en considération dans le calcul du codegré du produit des quadriques. Qu'est-ce qu'on met dans cette section? Une conjecture? Une autre conjecture qu'on pourrait mettre est qu'une certaine formule pour la duale en fonction du groupe de Weyl petit existe dès qu'on sait qu'il s'agit d'une hypersurface}
\medskip

\subsection{Counterexamples}
\label{sec_counterexamples}
The reader may wonder whether the hypothesis that $\rho_1(G_0)$ is semisimple is really important. 
The simplest example for which it is not fulfilled is, with $m=2$, given by the gradings of type 
$$\fsl_{a+b}=\fsl_a\times\fsl_b\times\CC\oplus (\CC^a\otimes\CC^b)\oplus (\CC^a\otimes\CC^b)^*.$$
Here $\fg_1=(\CC^a\otimes\CC^b)\oplus (\CC^a\otimes\CC^b)^*$ is not irreducible. Moreover the projectivization 
$\PP(\fg_1)$ contains two closed orbits, isomorphic to $\PP^{a-1}\times\PP^{b-1}$, so that their projective 
dual varieties (when considered in their linear spans) are not hypersurfaces in general (only for $a=b$).
In other words, everything goes wrong in that extremely simple example. 

\medskip
The reader may also wonder if we were right to restrict to the cases where $m=2,3$ or $5$. In the exceptional
cases there are indeed a few examples with $m=4$ and $m=6$ which we may briefly discuss. 
Assuming as before that 
$\rho_1(G_0)$ is semisimple there is only one case with $m=6$ to consider, which corresponds to the triple 
node of the affine Dynkin diagram $E_8^{(1)}$. Since the three arms of this diagram have lengths $1,2,5$, 
the unique closed orbit in $\PP(\fg_1)$ is $\PP^1\times\PP^2\times\PP^5$, in its Segre embedding. 
But for a Segre product of projective spaces $\PP^{n_1}\times\PP^{n_2}\times\cdots\times \PP^{n_r}$, 
with $n_1\le n_2\le\cdots\le n_r$, the projective dual is a hypersurface if and only if $n_r\le n_1+n_2+\cdots +n_{r-1}$
\cite[Corollary 5.10]{GKZ}. In particular the projective dual of $\PP^1\times\PP^2\times\PP^5$ is not a hypersurface!

\smallskip 
Another interesting case with $m=4$ is again attached to $E_8^{(1)}$, this time to the second node on its long arm,
starting from the triple node. In this case the closed orbit in $\PP(\fg_1)$ is $OG(5,10)_+\times\PP^3$. Even though in this case the $\theta$-$\corank$ vanishes, a computation shows that: 

\begin{prop}
The projective dual of $OG(5,10)_+\times\PP^3$ is not a hypersurface.
\end{prop}

\proof We apply the Katz-Kleiman formula for the degree (as a hypersurface) of the dual degree
of a projecive variety $X$. A straightforward consequence is that for any $k$, 
$$\deg (X\times\PP^k)^\vee=(k+1)\int_X\frac{\lambda^kc(\Omega_X)}{(1-\lambda)^{k+2}}=(k+1)\deg (X\cap\PP^{-k})^\vee,$$
where $\PP^{-k}$ means a general codimension $k$ linear space. 

Applying this to $X=OG(5,10)_+$ and $k=3$, we get a formula that we can reduce to an intersection 
product on $G(5,10)$. The result turns our to be zero, which exactly means that the projective dual of 
$OG(5,10)_+\times\PP^3$ is not a hypersurface.\qed

\section{Lagrangian Grassmannians, spinor varieties,  and perspectives} 
\label{sec_spin_lagr_varieties}

As observed in \cite{lascoux}, the dual degree of a Grassmannian can in principle be computed 
with the help of Schubert calculus (see also \cite{fnr} for a more general approach based on equivariant
Schubert calculus). This observation also applies to spinor varieties and Lagrangian Grassmannians $LG(n,2n)$ 
up to some minor modifications. 

\subsection{Lagrangian Grassmannians} 
Let us consider the Lagrangian Grassmannian $LG(n,2n)$ parametrizing Lagrangian subspaces in $\CC^{2n}$ endowed 
with some symplectic form. Recall that the cotangent 
bundle of this variety is just $S^2E$, where $E$ denotes the tautological rank $n$ vector bundle. Moreover,
as a subvariety of the usual Grassmannian $G(n,2n)$, the Lagrangian Grassmannian is defined by a general section 
of $\wedge^2E^*$ (we use the same notation for the tautological bundle on the Grassmannian and its restriction
to $LG(n,2n)$. Its fundamental class in the Chow ring of $G(n,2n)$ is therefore given by the Thom-Porteous 
formula, namely 
$$[LG(n,2n)]=c_{top}(\wedge^2E^*)=\sigma_{\delta(n)},$$
the Schubert cycle defined by the partition $\delta(n)=(n-1,\ldots , 2,1,0)$. Applying the Katz-Kleiman 
formula we get 
$$\deg (LG(n,2n)^\vee)=\int_{LG(n,2n)}\frac{c(S^2E)}{(1-\lambda)^2}=\int_{G(n,2n)}\frac{c(S^2E)\sigma_{\delta(n)}}{(1-\lambda)^2}.$$
This reduces the computation to a question of Schubert calculus, as announced. There are several difficulties. First we need to 
be able to decompose the total Chern class of the symmetric square of a vector bundle in terms of its Schur classes; in other
words, we need to compute (some of) the coefficients in the universal expansion 
$$c(S^2E)=\sum_{\ell(\mu)\le n}a_\mu s_\mu(E).$$
This problem was essentially solved in \cite{llt}. Plugging this into our previous formulas we get 
$$\deg (LG(n,2n)^\vee)=\sum_{k\ge 0}(k+1)
\sum_{\ell(\mu)\le n}(-1)^{|\mu|}a_\mu \int_{G(n,2n)}\lambda^k\sigma_\mu\sigma_{\delta(n)}.$$
In order to compute this integral, we need to decompose 
$$\lambda^k\sigma_\mu=\sum_{|\nu|=|\mu|+k}K_{\mu,\nu}\sigma_\nu,$$
where $K_{\mu,\nu}$ is the usual Kostka-Foulkes coefficient, and observe that by the duality properties
of the Schubert cycles, $\nu$ gives a non zero contribution only when it coincides with the complement of 
$\delta(n)$ in the $n\times n$ square, that is $\delta(n+1)$. Hence the expression:

\begin{prop}
$$\deg (LG(n,2n)^\vee)=
\sum_{\ell(\mu)\le n}(-1)^{|\mu|}\Bigg(\binom{n+1}{2}-|\mu|+1\Bigg)a_\mu K_{\mu,\delta(n+1)}.$$
\end{prop}

The classical properties of Kostka-Foulkes coefficients allow to restrict this sum to partitions $\mu\subset \delta(n+1)$, 
but this is still a large sum, with alternating signs. 

\medskip\noindent {\it Example}. For $n=3$ the answer is  already known: the projective dual of $LG(3,6)$ is a quartic
(so that $LG(3,6)$ belongs to the conjecturally short list of smooth varieties of dual degree four). For $n=4$
the previous expression is over $42$ partitions. The result is as expected,
$$\deg (LG(4,8)^\vee)=72.$$

\subsection{Spinor varieties}
The case of the spinor variety $\SS_{2n}$, parametrizing one of the families of maximal isotropic subspaces 
in $\CC^{2n}$, now endowed with a non-degenerate quadratic form, is very similar. The cotangent 
bundle of this variety is now $\wedge^2E$, where $E$ denotes again the tautological rank $n$ vector bundle. Moreover,
as a subvariety of the usual Grassmannian $G(n,2n)$, the Lagrangian Grassmannian is defined by a general section 
of $S^2E^*$, so its fundamental class in the Chow ring of $G(n,2n)$ is 
$$[\SS_{2n}]=c_{top}(S^2E^*)=2^n\sigma_{\delta(n+1)}.$$
In order to apply the Katz-Kleiman formula we need to recall that the Pl\"ucker line bundle is twice the generator
of the Picard group. We get 
$$\deg (\SS_{2n}^\vee)=\int_{\SS_{2n}}\frac{c(\wedge^2E)}{(1-\lambda/2)^2}=2^n\int_{G(n,2n)}\frac{c(\wedge^2E)\sigma_{\delta(n+1)}}{(1-\lambda/2)^2}.$$
Then we need to be able to decompose the total Chern class of the skew-symmetric square of a vector bundle in terms of its 
Schur classes; again there is a universal expansion 
$$c(\wedge^2E)=\sum_{\ell(\mu)\le n}b_\mu s_\mu(E),$$
also discussed in  \cite{llt}.
Plugging this into our previous formulas we get 
$$\deg (\SS_{2n}^\vee)=\sum_{k\ge 0}(k+1)
\sum_{\ell(\mu)\le n}(-1)^{|\mu|}b_\mu \int_{G(n,2n)}2^{n-k}\lambda^k\sigma_\mu\sigma_{\delta(n+1)}.$$
This finally yields the expression:

\begin{prop}
$$\deg (\SS_{2n}^\vee)= 2^{-\frac{n(n-3)}{2}}
\sum_{\ell(\mu)\le n}(-1)^{|\mu|}\Bigg(\binom{n}{2}-|\mu|+1\Bigg)2^{|\mu|}b_\mu K_{\mu,\delta(n)}.$$
\end{prop}
%\medskip 

\subsection{Perspectives}

In this paper we were able to extend Tevelev's formula to graded Lie algebras with prime grading and maximal $\theta$-$\rank$. Even though we have showed that the hypothesis we have made are somehow necessary (see Section \ref{sec_counterexamples}), one may still wonder whether a suitable version of Theorem \ref{thm_main} may hold in greater generality. 

Indeed, the proof of Theorem \ref{thm_main} is clearly divided in two parts, the first being the proof of the inclusion in the divisor, and the second being the dimension bound. We believe that the former could hold in a more general setting, for instance when the hypothesis of maximal $\theta$-$\rank$ is dropped. However, in order to deal with such a situation, one should be able to understand more 
deeply the structure of graded Lie algebras; for instance, it would be necessary to deduce a homogeneous decomposition of the algebra which resembles the one we have constructed in this paper (e.g. a decomposition in which a $\theta$-torus appears as a factor of an explicit Cartan subalgebra, so that the complex reflections in the little Weyl group can be made explicit). 

This generalization, together with the research for a general explicit formula for the dual of very classical homogeneous varieties (such as orthogonal and symplectic Grassmannians), will be the subject of future work.

\begin{comment}
\section{An extension of Tevelev's formula} 

There are essentially two situations for which the degree of the dual variety to a generalized Grassmannian 
$X=G/P\subset\PP (V)$ is known. The first one is when $\PP(V)$ admits only finitely many orbits: then the same
occurs for $\PP(V^\vee)$ and it is not difficult to identify among the finitely many orbits, which one is $X^\vee$ 
(usually the unique invariant hypersurface). The second one is when $X$ is an adjoint variety and we can apply 
Tevelev's formula. 

Lascoux studied in \cite{lascoux} the case of Grassmannians in their Plücker embeddings and reduced the computation
of their codegree to a (complicated)  intersection theory computation. He found that the codegree of $G(4,8)$ is 
$126$, which is also the number of roots in the exceptional root system $E_7$. Is that a coincidence? Of course not!

\laurent{Can we extend Tevelev's formula for the dual hypersurface of the adjoint variety to the case
of graded Lie algebras? The codegree of $G(3,9)$ is $120$ and there are $40$ reflection hyperplanes for $G_{32}$. 
The codegree of $G(4,8)$ is $126$ and there are $63$ reflection hyperplanes for $W(E_7)$. The length of the grading 
is three in the first case and two in the second one.. Then by analogy the codegree of the spinor variety $S_{16}$ should be $240$. Can we extend Lascoux's work on the codegrees of Grassmannians to spinor varieties? } 

\end{comment}

\bibliographystyle{alpha}

\bigskip

\bigskip
\noindent
\textsc{Institut de Mathématiques de Toulouse, UMR 5219,  Universit\'e Paul Sabatier, F-31062 Toulouse Cedex 9, France}

\noindent
\textit{Email address}: \texttt{manivel@math.cnrs.fr}

\medskip
\noindent
\textsc{Institut de Mathématiques de Bourgogne, UMR CNRS 5584, Universit\'e de Bourgogne et Franche-Comt\'e, 9 Avenue Alain Savary, BP 47870, 21078 Dijon Cedex, France}

\noindent
\textit{Email address}: \texttt{Vladimiro.Benedetti@u-bourgogne.fr  }

\end{document}